\documentclass{article}

\usepackage{amssymb}
\usepackage{hyperref}
\usepackage[leqno]{amsmath}
\usepackage{graphicx}

\usepackage{cases}
\usepackage{subfigure}
\usepackage{amsthm}
\usepackage[german,english]{babel}
\usepackage{color}
\usepackage{epstopdf}
\numberwithin{equation}{section}

\begin{document}
\title{A note of pointwise estimates on Shishkin meshes}

\author{%
Jin Zhang\thanks{Email: JinZhangalex@hotmail.com}
\footnote{Address:School of Science, Xi'an Jiaotong University,
Xi'an, 710049, China}}
\maketitle
\begin{abstract}
We propose the estimates of the discrete Green function for the streamline
diffusion finite element method (SDFEM) on Shishkin meshes.
\end{abstract}
\section{Problem}
We consider the singularly perturbed boundary value problem
\begin{align}\label{model problem}
 Lu:=-\varepsilon\Delta u+\boldsymbol{b}\cdot \nabla u+u&=f & \mathrm{in}& & &\Omega=(0,1)^{2},\tag{1.1a}\\
 u&=0 & \mathrm{on}& & &\partial\Omega,\tag{1.1b}
\end{align}
where $\varepsilon\ll 1$ is a small positive parameter and $\boldsymbol{b}=(b_{1},b_{2})^{T}>(0,0)^{T}$ is
constant. It is also assumed that $f$ is sufficiently smooth.

\section{The SDFEM on Shishkin meshes}

\subsection{Shishkin meshes }
Let $N>4$ be a positive even integer. We use a piecewise uniform mesh --- a so-called \textit{Shishkin} mesh ---with $N$ mesh intervals in both $x-$ and $y-$direction which condenses in the layer regions. For this purpose we define the two mesh transition parameters
\begin{equation*}
\lambda_{x}:=\min\left\{ \frac{1}{2},2\frac{\varepsilon}{\beta_{1}}\ln N \right\} \quad \mbox{and} \quad
\lambda_{y}:=\min\left\{
\frac{1}{2},2\frac{\varepsilon}{\beta_{2}}\ln N \right\}.
\end{equation*}
\newtheorem{assumption}{Assumption}
\begin{assumption}
We assume in our analysis that $\varepsilon\le N^{-1}$, as is generally the case in practice. Furthermore we assume that
$\lambda_{x}=2\varepsilon\beta^{-1}_{1}\ln N$ and $\lambda_{y}=2\varepsilon\beta^{-1}_{2}\ln N$ as otherwise $N^{-1}$ is exponentially small compared with $\varepsilon$.
\end{assumption}

\begin{figure}[h]
\centering
\includegraphics[width=0.8\textwidth]{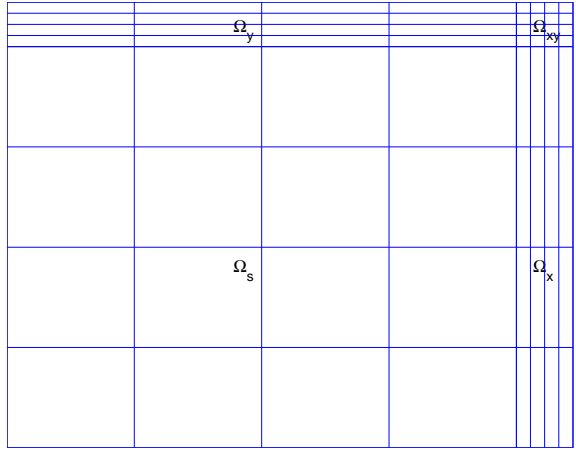}
\caption{Dissection of $\Omega$ and triangulation $\Omega^{N}$.}
\label{Shishkin mesh}
\end{figure}
The domain $\Omega$ is dissected into four parts as $\Omega=\Omega_{s}\cup\Omega_{x}\cup\Omega_{y}\cup\Omega_{xy}$(see FIG. \ref{Shishkin mesh}), where
\begin{align*}
&\Omega_{s}:=\left[0,1-\lambda_{x}\right]\times\left[0,1-\lambda_{y}\right],&&
\Omega_{x}:=\left[ 1-\lambda_{x},1 \right]\times\left[0,1-\lambda_{y}\right],\\
&\Omega_{y}:=\left[0,1-\lambda_{x}\right]\times\left[1-\lambda_{y},1 \right],&&
\Omega_{xy}:=\left[ 1-\lambda_{x},1 \right]\times\left[1-\lambda_{y},1 \right].
\end{align*}

\par
We introduce the set of mesh points $\left\{ (x_{i},y_{j})\in\Omega:i,\,j=0,\,\cdots,\,N  \right\}$ defined by
\begin{numcases}{x_{i}=}
2i(1-\lambda_{x})/N ,&\text{for $i=0,\,\cdots,\,N/2$}, \nonumber\\
1-2(N-i)\lambda_{x}/N, &\text{for $i=N/2+1,\,\cdots,\,N$}\nonumber
\end{numcases}
and
\begin{numcases}{y_{j}=}
2j(1-\lambda_{y})/N ,&\text{for $j=0,\,\cdots,\,N/2$}, \nonumber\\
1-2(N-j)\lambda_{y}/N, &\text{for $j=N/2+1,\,\cdots,\,N$}.\nonumber
\end{numcases}
By drawing lines through these mesh points parallel to the $x$-axis and $y$-axis the domain $\Omega$ is partitioned into rectangles. This triangulation is denoted by $\Omega^{N}$(see FIG. \ref{Shishkin mesh}). If $D$ is a mesh subdomain of $\Omega$, we write $D^{N}$ for the triangulation of $D$. The mesh sizes $h_{x,\tau}=x_{i}-x_{i-1}$ and $h_{y,\tau}=y_{j}-y_{j-1}$ satisfy
\begin{numcases}{h_{x,\tau}=}
H_{x}:=\frac{1-\lambda_{x}}{N/2}, &\text{for $i=1,\,\cdots,\,N/2$}, \nonumber\\
h_{x}:=\frac{\lambda_{x}}{N/2}, &\text{for $i=N/2+1,\,\cdots,\,N$} \nonumber
\end{numcases}
and
\begin{numcases}{h_{y,\tau}=}
H_{y}:=\frac{1-\lambda_{y}}{N/2}, &\text{for $j=1,\,\cdots,\,N/2$},\nonumber \\
h_{y}:=\frac{\lambda_{y}}{N/2}, &\text{for $j=N/2+1,\,\cdots,\,N$}.\nonumber
\end{numcases}
The mesh sizes $h_{x,\tau}$ and $h_{y,\tau}$ satisfy
\begin{equation*}
N^{-1}\le H_{x},H_{y} \le 2N^{-1} \quad \mbox{and} \quad
C_{1}\varepsilon N^{-1}\ln N \le h_{x},h_{y}\le C_{2}\varepsilon N^{-1}\ln N,
\end{equation*}
where $C_{1}$ and $C_{2}$ are positive constants and independent of $\varepsilon$
and of the mesh parameter $N$. The above properties are essential when inverse inequalities
are applied in our later analysis.
\par
For the mesh elements we shall use two notations: $\tau_{ij}=[x_{i-1},x_{i}]\times [y_{j-1},y_{j}]$ for a specific element, and $\tau$ for a generic mesh rectangle.
\subsection{The streamline diffusion finite element method}

Let $V:=H^{1}_{0}(\Omega)$. On the above Shishkin mesh we define a finite element space
\begin{equation*}
V^{N}:=\{v^{N}\in C(\bar{\Omega}):v^{N}|_{\partial\Omega}=0
 \text{ and $v^{N}|_{\tau}$ is bilinear, }
  \forall\tau\in\Omega^{N} \}.
\end{equation*}
\par
 In this case, the SDFEM reads as
\begin{equation}\label{SDFEM}
\left\{
\begin{array}{lr}
\text{Find $U\in V^{N}$ such that for all $v^{N}\in V^{N}$}\\
\varepsilon(\nabla U,\nabla v^{N})+(\boldsymbol{b}\cdot\nabla U+U,v^{N}+\delta\boldsymbol{b}\cdot\nabla v^{N})=(f,v^{N}+\delta\boldsymbol{b}\cdot\nabla v^{N}).
\end{array}
\right.
\end{equation}
where $\delta=\delta(\boldsymbol{x})$ is a
user-chosen parameter (see \cite{Johnson:1987-Numerical}).
\par
We set
\begin{equation*}
    b:=\sqrt{b^{2}_{1}+b^{2}_{2}},\quad
    \boldsymbol{\beta}:=\genfrac(){0cm}{0}{b_{1}}{b_{2}}/b,\quad
    \boldsymbol{\eta}:=\genfrac(){0cm}{0}{-b_{2}}{b_{1}}/b
     \quad\mbox{and}\quad
    v_{\zeta}:=\boldsymbol{\zeta}^{T}\nabla v
\end{equation*}
for any vector $\boldsymbol{\zeta}$ of unit length. By an easy calculation one shows that
\begin{equation*}
(\nabla w,\nabla v)=(w_{\beta},v_{\beta})+(w_{\eta},v_{\eta}).
\end{equation*}
We rewrite \eqref{SDFEM} as
\begin{equation*}
\varepsilon(U_{\beta},v^{N}_{\beta})+\varepsilon(U_{\eta},v^{N}_{\eta})
+(bU_{\beta}+U,v^{N}+\delta b v^{N}_{\beta})=
(f,v^{N}+\delta b v^{N}_{\beta})
\end{equation*}
and, following usual practice, we set
\begin{numcases}{\delta(\boldsymbol{x}):=}
N^{-1},&\text{if $\boldsymbol{x}\in\Omega_{s}$},\nonumber\\
0,&\text{otherwise}.\nonumber
\end{numcases}
For technical reasons in the later analysis, we increase the crosswind diffusion(see \cite{Joh1Sch2Wah3:1987-Crosswind}) by replacing $\varepsilon(U_{\eta},v^{N}_{\eta})$ by
$\hat{\varepsilon}(U_{\eta},v^{N}_{\eta})$ where
\begin{equation*}
\tilde{\varepsilon}:=\max(\varepsilon,N^{-3/2})
\end{equation*}
and
\begin{numcases}{\hat{\varepsilon}(\boldsymbol{x}):=}
\tilde{\varepsilon},&\text{ $\boldsymbol{x}\in\Omega_{s}$},\nonumber\\
\varepsilon,&\text{$\boldsymbol{x}\in\Omega \setminus\Omega_{s}$}.\nonumber
\end{numcases}

\par
We now state our streamline diffusion method with artificial crosswind:
\begin{equation}\label{SDFEM with ACD}
\left\{
\begin{array}{lr}
\text{Find $U\in V^{N}$ such that for all $v^{N}\in V^{N}$}\\
B(U,v^{N})=(f,v^{N}+\delta bv^{N}_{\beta}),
\end{array}
\right.
\end{equation}
with
\begin{equation}\label{bilinear SDFEM with ACD}
B(U,v^{N}):=(\varepsilon+b^{2}\delta)(U_{\beta},v^{N}_{\beta})+
\hat{\varepsilon}(U_{\eta},v^{N}_{\eta})-b(1-\delta)(U,v^{N}_{\beta})+(U,v^{N}).
\end{equation}

\section{The discrete Green function}
Let $\boldsymbol{x}^{\ast}$ be a mesh node in $\Omega$. The discrete Green's function $G\in V^{N}$ associated with $\boldsymbol{x}^{\ast}$ is defined by

\begin{equation*}
B(v^{N},G)=v^{N}(\boldsymbol{x}^{\ast}),\,\forall v^{N}\in V^{N}.
\end{equation*}

\par\noindent
The weighted function $\omega$:
\begin{equation*}
\omega(\boldsymbol{x}):=
g\left(\frac{(\boldsymbol{x}-\boldsymbol{x}^{\ast})\cdot\boldsymbol{\beta}}{\sigma_{\beta}}\right)
g\left(\frac{(\boldsymbol{x}-\boldsymbol{x}^{\ast})\cdot\boldsymbol{\eta}}{\sigma_{\eta}}\right)
g\left(-\frac{(\boldsymbol{x}-\boldsymbol{x}^{\ast})\cdot\boldsymbol{\eta}}{\sigma_{\eta}}\right)
\end{equation*}
where
\begin{equation*}
g(r)=\frac{2}{1+e^{r}}\quad\text{ for $r\in(-\infty,\infty)$}.
\end{equation*}
and $\sigma_{\beta}=kN^{-1}\ln N$ and $\sigma_{\eta}=k\tilde{\varepsilon}^{1/2} \ln N$.
\begin{equation*}
\vert\vert\vert G \vert\vert\vert^{2}_{\omega}:=(\varepsilon+b^{2}\delta)\Vert \omega^{-1/2}G_{\beta} \Vert^{2}
+\hat{\varepsilon}\Vert \omega^{-1/2}G_{\eta} \Vert^{2}+\frac{b}{2}\Vert (\omega^{-1})^{1/2}_{\beta}G\Vert^{2}+\Vert \omega^{-1/2}G \Vert^{2}
\end{equation*}
and
\begin{align}\label{analysis equation}
\vert\vert\vert G \vert\vert\vert^{2}_{\omega}
=&B(\omega^{-1}G,G)-(\varepsilon+b^{2}\delta)((\omega^{-1})_{\beta}G,G_{\beta})\\
&-\hat{\varepsilon}((\omega^{-1})_{\eta}G,G_{\eta})-b\delta(\omega^{-1}G,G_{\beta})\nonumber.
\end{align}
Thus, we obtain
\begin{align*}
B(\omega^{-1}G,G)&=B(E,G)+B((\omega^{-1}G)^{I},G)\\
&=B(E,G)+(\omega^{-1}G)(\boldsymbol{x}^{\ast})\nonumber
\end{align*}
where $E:=\omega^{-1}G-(\omega^{-1}G)^{I}$.
\newtheorem{lemma}{Lemma}
\begin{lemma}
If $\sigma_{\beta}=kN^{-1}\ln N$ and $\sigma_{\eta}=k\ln N\tilde{\varepsilon}^{1/2}$, then for $k>1$ sufficiently
large and independent of $N$ and $\varepsilon$, we have
\begin{equation*}
B(\omega^{-1}G,G)\ge \frac{1}{4}\Vert\vert G \vert\Vert^{2}_{\omega}.
\end{equation*}
\end{lemma}
\begin{proof}
From \eqref{analysis equation},we estimate the following terms.
\begin{align*}
(\varepsilon+\delta)\left|((\omega^{-1})_{\beta}G,G_{\beta})\right|
&\le
C(\varepsilon+\delta)^{1/2}\sigma^{-1/2}_{\beta}\cdot
\Vert (\omega^{-1})^{1/2}_{\beta}G \Vert
\cdot(\varepsilon+\delta)^{1/2}\Vert \omega^{-1/2}G_{\beta} \Vert\\
&\le
C(\varepsilon+\delta)^{1/2}\sigma^{-1/2}_{\beta}
\Vert\vert G \vert\Vert^{2}_{\omega}
\end{align*}
and
\begin{align*}
\hat{\varepsilon}\left|((\omega^{-1})_{\eta}G,G_{\eta})\right|
&\le
C\hat{\varepsilon}^{1/2}\sigma^{-1}_{\eta}\cdot
\Vert \omega^{-1/2}G \Vert
\cdot\hat{\varepsilon}^{1/2}\Vert \omega^{-1/2}G_{\eta} \Vert\\
&\le
C\hat{\varepsilon}^{1/2}\sigma^{-1}_{\eta}
\Vert\vert G \vert\Vert^{2}_{\omega}
\end{align*}
For $b\delta(\omega^{-1}G,G_{\beta})$, we make use of integration by parts.
\par
From the definition of $\sigma_{\beta}$ and $\sigma_{\eta}$ and $\varepsilon\le N^{-1}$, we take $k$ sufficiently large and we are done.
\end{proof}

\begin{lemma}
If  $\sigma_{\beta}=kN^{-1}\ln N$, with $k>0$ sufficiently
large and independent of $N$ and $\varepsilon$. Then for each mesh point $\boldsymbol{x}^{\ast}\in\Omega\setminus\Omega_{xy}$, we have
\begin{equation*}
\left|(\omega^{-1}G)(\boldsymbol{x}^{\ast})\right|
\le \frac{1}{16}\Vert\vert G \vert\Vert^{2}_{\omega}+CN\ln N.
\end{equation*}
where $C$ is independent of $N$, $\varepsilon$ and $\boldsymbol{x}^{\ast}$.
\end{lemma}
\begin{proof}
First let $\boldsymbol{x}^{\ast}\in\Omega_{s}$. Let $\tau^{\ast}$ be the unique triangle that has $\boldsymbol{x}^{\ast}$ as its north-east corner. Then
\begin{align*}
\left| (\omega^{-1}G)(\boldsymbol{x}^{\ast})\right|
&\le CN\Vert G \Vert_{\tau^{\ast}}\\
&\le
CN\underset{\tau^{\ast}}{\max}
\left|(\omega^{-1})^{-1/2}_{\beta}\right|\cdot
\Vert (\omega^{-1})^{-1/2}_{\beta}G \Vert_{\tau^{\ast}}
\end{align*}
Calculating $(\omega^{-1})^{-1}_{\beta}(\boldsymbol{x})$ explicitly, we see that
\begin{equation*}
(\omega^{-1})^{-1}_{\beta}(\boldsymbol{x})
\le
C\sigma_{\beta}=CkN^{-1}\ln N\quad\forall\boldsymbol{x}\in\tau^{\ast}
\end{equation*}
Thus
\begin{equation*}
\left| (\omega^{-1}G)(\boldsymbol{x}^{\ast})\right|
\le
CN\ln N+\frac{1}{16}\Vert\vert G \vert\Vert^{2}_{\omega}
\end{equation*}
by means of the arithmetic-geometric mean inequality.
\par
Next, let $\boldsymbol{x}^{\ast}\in\Omega_{x}$.(The case $\boldsymbol{x}^{\ast}\in\Omega_{y}$ is similar.)Write
$\boldsymbol{x}^{\ast}=(x_{i},y_{j})$.
Then
\begin{align*}
\left|\omega^{-1}G(\boldsymbol{x}^{\ast})\right|
&=
\left|G(\boldsymbol{x}^{\ast})\right|\\
&=\left|\int^{1}_{x_{i}}G_{x}(t,y_{j})\mathrm{d}t\right|\\
&\le
CH^{-1}_{y}\int^{1}_{x_{i}}\int^{y_{j+1}}_{y_{j}}
\left| G_{x}(t,y) \right|\mathrm{d}y\mathrm{d}t\\
&\le
CN(\varepsilon\ln N\cdot N^{-1})^{1/2}\Vert G_{x} \Vert_{\Omega_{x}}\\
&\le
CN^{1/2}\ln^{1/2}N\Vert\vert G \vert\Vert
\end{align*}
where $G_{x}(t,y_{j})=\frac{G_{k,j}-G_{k-1,j}}{h_{x}}$
for $(t,y_{j})\in\tau_{kj}$.\\*
Analysis: for the relation of boundary integral and domain integral, we analyze
\begin{equation*}
\int^{1}_{x_{i}}\int^{y_{j+1}}_{y_{j}}
\left| G_{x}(t,y) \right|\mathrm{d}y\mathrm{d}t
=
\sum^{N}_{k=i+1}h_{x}
\int^{y_{j+1}}_{y_{j}}
\left|f(y_{j})\frac{y_{j+1}-y}{H_{y}}+f(y_{j+1})\frac{y-y_{j}}{H_{y}}\right|\mathrm{d}y
\end{equation*}
where $f(y_{j})=\frac{G_{k,j}-G_{k-1,j}}{h_{x}}$ and
$f(y_{j+1})=\frac{G_{k,j+1}-G_{k-1,j+1}}{h_{x}}$.\\*
For $\Delta_{1}:=\int^{y_{j+1}}_{y_{j}}
\left|f(y_{j})\frac{y_{j+1}-y}{H_{y}}+f(y_{j+1})\frac{y-y_{j}}{H_{y}}\right|\mathrm{d}y$, we have
\begin{numcases}{\Delta_{1}=}
\frac{|f(y_{j})|+|f(y_{j+1})|}{2}H_{y}\ge\frac{1}{2}\max\{|f(y_{j}),|f(y_{j+1})|\}H_{y}&\text{if $f(y_{j})f(y_{j+1})\ge0$}\nonumber\\
\frac{1}{2}\frac{f^{2}(y_{j})+f^{2}(y_{j+1})}{|f(y_{j+1})-f(y_{j})|}H_{y}
\ge\frac{1}{4}\max\{|f(y_{j}),|f(y_{j+1})|\}H_{y}
&\text{if $f(y_{j})f(y_{j+1})<0$ }\nonumber
\end{numcases}

\end{proof}
%
%

\par
For $\forall v\in C^{2}(\tau)$, we have
\begin{equation*}
|v_{x}|\le C(|v_{\beta}|+|v_{\eta}|)
\end{equation*}
\begin{equation*}
|v_{xx}|\le C(|v_{\beta\beta}|+|v_{\beta\eta}|+|v_{\eta\eta}|)
\end{equation*}
Similarly, we have
\begin{equation*}
|v_{\beta}|\le C(|v_{x}|+|v_{y}|)
\end{equation*}
\begin{equation*}
|v_{\beta\beta}|\le C(|v_{xx}|+|v_{xy}|+|v_{yy}|)
\end{equation*}
\begin{lemma}\label{lemma derivative of E}
Let $\tau\in\Omega^{N}$. Then
\begin{align*}
\Vert \omega^{1/2}D^{\alpha}E \Vert_{\Omega_{s}}&\le Ck^{-1/2}N^{1/2}\vert\vert\vert G \vert\vert\vert_{\omega}\\
\Vert \omega^{1/2}D^{\alpha}E  \Vert_{\Omega^{N}\setminus\Omega_{s}}&\le Ck^{-1}\varepsilon^{-1/2}\ln^{-1}N\vert\vert\vert G \vert\vert\vert_{\omega}
\end{align*}
where $H_{\tau}=\max\{h_{x,\tau},h_{y,\tau}\}$ and  $|\alpha|=1$, $D^{\alpha}G$ are
$G_{\beta}$ and $G_{\eta}$.
\end{lemma}
\begin{proof}
Assume $p\in[1,\infty]$ and $g\in C^{3}(\tau)$. Then (see \cite[Theorem 4]{Apel1Dobr2:1992-Anisotropic})
\begin{align*}
\Vert (g-g^{I})_{x} \Vert_{L^{p}(\tau)} &\le
C\left( h^{2}_{x,\tau}\Vert g_{xxx} \Vert_{L^{p}(\tau)}+ h_{x,\tau}h_{y,\tau}\Vert g_{xxy} \Vert_{L^{p}(\tau)} +
h^{2}_{y,\tau}\Vert g_{xyy} \Vert_{L^{p}(\tau)}
\right)\\
&+Ch_{x,\tau}\Vert g_{xx} \Vert_{L^{p}(\tau)},\\
\Vert (g-g^{I})_{y} \Vert_{L^{p}(\tau)} &\le
C\left( h^{2}_{x,\tau}\Vert g_{xxy} \Vert_{L^{p}(\tau)}+ h_{x,\tau}h_{y,\tau}\Vert g_{xyy} \Vert_{L^{p}(\tau)} +
h^{2}_{y,\tau}\Vert g_{yyy} \Vert_{L^{p}(\tau)}
\right)\\
&+Ch_{y,\tau}\Vert g_{yy} \Vert_{L^{p}(\tau)}.
\end{align*}
or (see \cite[Comment 2.15]{Apel:1999-Anisotropic})
\begin{align*}
\Vert (g-g^{I})_{x} \Vert_{L^{2}(\tau)} &\le
Ch^{2}_{y,\tau}\Vert g_{xyy} \Vert_{L^{2}(\tau)}
+Ch_{x,\tau}\Vert g_{xx} \Vert_{L^{2}(\tau)},\\
\Vert (g-g^{I})_{y} \Vert_{L^{2}(\tau)} &\le
Ch^{2}_{x,\tau}\Vert g_{xxy} \Vert_{L^{2}(\tau)}+Ch_{y,\tau}\Vert g_{yy} \Vert_{L^{2}(\tau)}.
\end{align*}
In the following analysis, $D^{\alpha}v$ denotes the directional derivative of $v$ along $\beta$ or $\eta$ for different orders. The following analysis makes use of the former estimates (The latter will make the analysis more shorter).
\par
For $\tau\in\Omega^{N}$, we have
\begin{align*}
&\Vert \omega^{1/2}E_{\beta} \Vert_{\tau}
\le
C\underset{\tau}{\max}\omega^{1/2}(\Vert E_{x} \Vert_{\tau}+\Vert E_{y} \Vert_{\tau})\\
&\le
C\underset{\tau}{\max}\omega^{1/2}\{
h^{2}_{x,\tau}\Vert(\omega^{-1}G)_{xxx}\Vert_{\tau}
+h_{x,\tau}\Vert(\omega^{-1}G)_{xx}\Vert_{\tau}
+h_{x,\tau}H_{\tau}\Vert(\omega^{-1}G)_{xxy}\Vert_{\tau}\\
&+
H_{\tau}h_{y,\tau}\Vert(\omega^{-1}G)_{xyy}\Vert_{\tau}
+h^{2}_{y,\tau}\Vert(\omega^{-1}G)_{yyy}\Vert_{\tau}
+h_{y,\tau}\Vert(\omega^{-1}G)_{yy}\Vert_{\tau}
\}\\
&\le
Ch^{2}_{x,\tau}(\Vert(\omega^{-1})_{xxx}G\Vert_{\tau}+\Vert(\omega^{-1})_{xx}G_{x}\Vert_{\tau})
+Ch_{x,\tau}(\Vert(\omega^{-1})_{xx}G\Vert_{\tau}+\Vert(\omega^{-1})_{x}G_{x}\Vert_{\tau})\\
&+Ch^{2}_{y,\tau}(\Vert(\omega^{-1})_{yyy}G\Vert_{\tau}+\Vert(\omega^{-1})_{yy}G_{y}\Vert_{\tau})
+Ch_{y,\tau}(\Vert(\omega^{-1})_{yy}G\Vert_{\tau}+\Vert(\omega^{-1})_{y}G_{y}\Vert_{\tau})\\
&+Ch_{x,\tau}H_{\tau}(\Vert(\omega^{-1})_{xxy}G\Vert_{\tau}+\Vert(\omega^{-1})_{xx}G_{y}\Vert_{\tau}
+\Vert(\omega^{-1})_{xy}G_{x}\Vert_{\tau}+\Vert(\omega^{-1})_{x}G_{xy}\Vert_{\tau})\\
&+H_{\tau}h_{y,\tau}(\Vert(\omega^{-1})_{xyy}G\Vert_{\tau}+\Vert(\omega^{-1})_{xy}G_{y}\Vert_{\tau}
+\Vert(\omega^{-1})_{yy}G_{x}\Vert_{\tau}+\Vert(\omega^{-1})_{y}G_{xy}\Vert_{\tau})\\
&\le
CH_{\tau}^{2}
\sum^{3}_{k=2}
\sum_{
\begin{array}{l}
\scriptscriptstyle |\alpha|+|\gamma|=3\\
\scriptscriptstyle |\alpha|\ge k
\end{array}
}
\Vert D^{\alpha}(\omega^{-1})D^{\gamma}G\Vert_{\tau}
+CH_{\tau}
\sum^{2}_{k=1}
\sum_{
\begin{array}{l}
\scriptscriptstyle |\alpha|+|\gamma|=2\\
\scriptscriptstyle |\alpha|\ge k
\end{array}
}
\Vert D^{\alpha}(\omega^{-1})D^{\gamma}G\Vert_{\tau}
\end{align*}
where we have used the following analysis for $\sum_{|\alpha|=1,|\gamma|=2}D^{\alpha}(\omega^{-1})D^{\gamma}G$ :
\begin{align*}
\Vert (\omega^{-1})_{x}G_{xy}\Vert_{\tau}
&\le
C\Vert \sum_{|\alpha|=1}|D^{\alpha}(\omega^{-1})|\cdot G_{xy} \Vert_{\tau}\\
&\le
C\underset{\tau}{\max} \sum_{|\alpha|=1}|D^{\alpha}(\omega^{-1})|\cdot\Vert G_{xy} \Vert_{\tau}\\
&\le
Ch^{-1}_{y,\tau}\underset{\tau}{\max} \sum_{|\alpha|=1}|D^{\alpha}(\omega^{-1})|\cdot\Vert G_{x} \Vert_{\tau}\\
&\le
Ch^{-1}_{y,\tau} \big\Vert \sum_{|\alpha|=1}|D^{\alpha}(\omega^{-1})|\cdot G_{x} \big\Vert_{\tau}.
\end{align*}
\par

The same analysis can be applied to $\Vert\omega^{1/2}E_{\eta}\Vert_{\tau}$.
\par
For $\tau\in\Omega_{s}$, we have
\begin{align*}
&\Vert \omega^{1/2}E_{\beta} \Vert_{\tau}
\le
Ck^{-5/2}N^{-2}
\big[(\sigma^{-5/2}_{\beta}+\sigma^{-3/2}_{\beta}\sigma^{-1}_{\eta}+\sigma^{-1/2}_{\beta}\sigma^{-2}_{\eta})
\underset{\tau}{\max}(\omega^{-1})^{1/2}_{\beta}
+\sigma^{-3}_{\eta}\underset{\tau}{\max}\omega^{-1/2}\big]\Vert G \Vert_{\tau}\\
&+Ck^{-3/2}H^{2}_{\tau}
\big[
(\sigma^{-3/2}_{\beta}+\sigma^{-1/2}_{\beta}\sigma^{-1}_{\eta})\underset{\tau}{\max}(\omega^{-1})^{1/2}_{\beta}
+\sigma^{-2}_{\eta}\underset{\tau}{\max}\omega^{-1/2}
\big]\cdot N\Vert G \Vert_{\tau}\\
&+Ck^{-3/2}H_{\tau}
\big[(\sigma^{-3/2}_{\beta}+\sigma^{-1/2}_{\beta}\sigma^{-1}_{\eta})
\underset{\tau}{\max}(\omega^{-1})^{1/2}_{\beta}
+\sigma^{-2}_{\eta}\underset{\tau}{\max}\omega^{-1/2}\big]\Vert G \Vert_{\tau}\\
&
+C\underset{\tau}{\max}\omega^{1/2}H_{\tau}
\sum_{|\alpha|=1}\Vert D^{\alpha}(\omega^{-1}) \Vert_{L^{\infty}(\tau)}\cdot \sum_{|\gamma|=1}\Vert D^{\gamma}G \Vert_{\tau}\\
&\le
Ck^{-1/2}N^{1/2}\vert\vert\vert G \vert\vert\vert_{\omega}
\end{align*}
where we have used the estimates of $\omega^{-1}$, standard inverse estimates and H{\"o}lder inequalities.
Similarly, we have
\begin{equation*}
\Vert \omega^{1/2}E_{\eta} \Vert_{\Omega_{s}}
\le
Ck^{-1/2}N^{1/2}\vert\vert\vert G \vert\vert\vert_{\omega}.
\end{equation*}
\par

For $\tau\in\Omega^{N}\setminus\Omega^{N}_{s}$, we have
\begin{align*}
&\Vert \omega^{1/2}E_{\beta} \Vert_{\tau}
\le
Ck^{-5/2}H^{2}_{\tau}
\big[(\sigma^{-5/2}_{\beta}+\sigma^{-3/2}_{\beta}\sigma^{-1}_{\eta}+\sigma^{-1/2}_{\beta}\sigma^{-2}_{\eta})
\underset{\tau}{\max}(\omega^{-1})^{1/2}_{\beta}
+\sigma^{-3}_{\eta}\underset{\tau}{\max}\omega^{-1/2}
\big]\Vert G \Vert_{\tau}\\
&+Ck^{-2}\varepsilon^{-1/2}H^{2}_{\tau}
\big[
\sigma^{-2}_{\beta}+\sigma^{-1}_{\beta}\sigma^{-1}_{\eta}
+\sigma^{-2}_{\eta}
\big]\underset{\tau}{\max}\omega^{-1/2}\cdot \varepsilon^{1/2}\sum_{|\gamma|=1}\Vert D^{\gamma}G \Vert_{\tau}\\
&+Ck^{-3/2}H_{\tau}
\big[
(\sigma^{-3/2}_{\beta}+\sigma^{-1/2}_{\beta}\sigma^{-1}_{\eta})
\underset{\tau}{\max}(\omega^{-1})^{1/2}_{\beta}
+\sigma^{-2}_{\eta}\underset{\tau}{\max}\omega^{-1/2}
\big]
\Vert G \Vert_{\tau}\\
&+
Ck^{-1}\varepsilon^{-1/2}H_{\tau}
\big[
\sigma^{-1}_{\beta}+\sigma^{-1}_{\eta}
\big]\underset{\tau}{\max}\omega^{-1/2}\cdot \varepsilon^{1/2}\sum_{|\gamma|=1}\Vert D^{\gamma}G \Vert_{\tau}\\
&\le
Ck^{-1}\varepsilon^{-1/2}\ln^{-1}N\vert\vert\vert G \vert\vert\vert_{\omega}
\end{align*}
Similarly, we have
\begin{equation*}
\Vert \omega^{1/2}E_{\eta} \Vert_{\Omega\setminus\Omega_{s}}
\le
Ck^{-1}\varepsilon^{-1/2}\ln^{-1}N\vert\vert\vert G \vert\vert\vert_{\omega}.
\end{equation*}
\end{proof}

%
%
\begin{lemma}\label{lemma E}
Let $\tau\in\Omega^{N}_{s}$. Let $E=\omega^{-1}G-(\omega^{-1}G)^{I}$ where $(\omega^{-1}G)^{I}$ denote the bilinear function that interpolates to $\omega^{-1}G$ at the vertices of $\tau$. Then
\begin{align*}
&\Vert \omega^{1/2}E \Vert_{\Omega_{s}} \le Ck^{-1}N^{-1/2}\vert\vert\vert G \vert\vert\vert_{\omega}\\
&\Vert \omega^{1/2}E \Vert_{\Omega\setminus\Omega_{s}}
\le
Ck^{-1}\varepsilon^{1/2}\vert\vert\vert G \vert\vert\vert_{\omega}
\end{align*}
where $H_{\tau}=\max\{h_{x,\tau},h_{y,\tau}\}$.
\end{lemma}
\begin{proof}
We make use of the following standard interpolation error bounds
\begin{equation*}
\Vert u-u^{I} \Vert_{L^{p}(\tau)}\le
 h^{2}_{x,\tau}\Vert u_{xx} \Vert_{L^{p}(\tau)}+
 h^{2}_{y,\tau}\Vert u_{yy} \Vert_{L^{p}(\tau)}
\end{equation*}
where $p\in [1,\infty]$ and $u\in C(\overline{\tau})\cap W^{2,p}(\tau)$.
\par
Then, we have
\begin{align*}
&\Vert E \Vert_{\tau}\le
h^{2}_{x,\tau}\Vert(\omega^{-1}G)_{xx}\Vert_{\tau}
+h^{2}_{y,\tau}\Vert(\omega^{-1}G)_{yy}\Vert_{\tau}
\\
&\le
CH^{2}_{\tau}  \sum_{|\alpha|=2}\Vert D^{\alpha}(\omega^{-1})\cdot G \Vert_{\tau}
+CH^{2}_{\tau}\Vert (|(\omega^{-1})_{\beta}|+|(\omega^{-1})_{\eta}|)\cdot (|G_{\beta}|+|G_{\eta}|) \Vert_{\tau}
\\
&\le
CH^{2}_{\tau}\left\{
\Vert(\omega^{-1})_{\beta}G_{\beta} \Vert_{\tau}+\Vert(\omega^{-1})_{\beta}G_{\eta} \Vert_{\tau}
+\Vert(\omega^{-1})_{\eta}G_{\beta} \Vert_{\tau}+\Vert(\omega^{-1})_{\eta}G_{\eta} \Vert_{\tau}
\right\}\\
&+
CH^{2}_{\tau} \Vert \sum_{|\alpha|=2}|D^{\alpha}(\omega^{-1})|\cdot G \Vert_{\tau}.
\end{align*}

\par
From the above inequality, we have
\begin{equation*}
\Vert \omega^{1/2}E \Vert_{\Omega_{s}}
\le
Ck^{-1}N^{-1/2}\vert\vert\vert G \vert\vert\vert_{\omega}
\end{equation*}
and
\begin{equation*}
\Vert \omega^{1/2}E \Vert_{\Omega\setminus\Omega_{s}}
\le
Ck^{-1}\varepsilon^{-1/2}N^{-1}\vert\vert\vert G \vert\vert\vert_{\omega}.
\end{equation*}

Following the techniques of (see\cite[Lemma 4.4]{Linb1Styn2:2001-SDFEM}), we have
\begin{equation*}
(\omega^{1/2}E)(\boldsymbol{x})=\int_{\boldsymbol{x}}^{\Gamma(\boldsymbol{x})}(\omega^{1/2}E)_{\eta}\mathrm{d}s
\end{equation*}
where $\boldsymbol{x}\in\Omega\setminus\Omega_{s}$, $\Gamma(\boldsymbol{x})\in \Gamma$ satisfies
  $(\boldsymbol{x}-\Gamma(\boldsymbol{x}))\cdot \beta=0$ and the following condition:
 \begin{equation*}
 \begin{array}{r}
 \text{For $\forall \boldsymbol{y}\in \Gamma$, $(\boldsymbol{x}-\boldsymbol{y})\cdot \beta=0$,}\\
 |\boldsymbol{x}-\Gamma(\boldsymbol{x})|=\underset{y}{\min} |\boldsymbol{x}-\boldsymbol{y}|.
 \end{array}
 \end{equation*}
From the above representation of $\omega^{1/2}E$, we have
\begin{align*}
&\Vert \omega^{1/2}E \Vert^{2}_{\Omega_{x}}=
\int^{1-\lambda_{y}}_{\lambda_{0}}\int_{l(\boldsymbol{x}_{lu},\Gamma(\boldsymbol{x}_{lu}))}
\left[\int_{\boldsymbol{x}}^{\Gamma(\boldsymbol{x})}(\omega^{1/2}E)_{\eta}\mathrm{d}s\right]^{2}\mathrm{d}\Omega\\
&+\int_{1-\lambda_{x}}^{1}\int_{l(\boldsymbol{x}_{u},\Gamma(\boldsymbol{x}_{u}))}
\left[\int_{\boldsymbol{x}}^{\Gamma(\boldsymbol{x})}(\omega^{1/2}E)_{\eta}\mathrm{d}s\right]^{2}\mathrm{d}\Omega
+\int_{0}^{\lambda_{0}}\int_{l(\boldsymbol{x}_{ld},\Gamma(\boldsymbol{x}_{ld}))}
\left[\int_{\boldsymbol{x}}^{\Gamma(\boldsymbol{x})}(\omega^{1/2}E)_{\eta}\mathrm{d}s\right]^{2}\mathrm{d}\Omega\\
&\le
C\lambda^{2}_{x}
\left\{\Vert(\omega^{1/2})_{\eta}E\Vert^{2}_{\Omega_{x}}+
 \Vert\omega^{1/2}E_{\eta}\Vert^{2}_{\Omega_{x}}
\right\}\\
&\le
C\varepsilon^{2}\ln^{2}N
\left\{\sigma^{-2}_{\eta}\Vert \omega^{1/2}E \Vert^{2}_{\Omega_{x}}+
 \Vert\omega^{1/2}E_{\eta}\Vert^{2}_{\Omega_{x}}
\right\}\\
&\le
Ck^{-2}\varepsilon^{2}\ln^{2}N\{N^{3/2}\ln^{-2}N \cdot\varepsilon^{-1}N^{-2}+
 \varepsilon^{-1}\ln^{-2}N \}
\vert\vert\vert G \vert\vert\vert^{2}_{\omega}\\
&\le
Ck^{-2}\varepsilon
\vert\vert\vert G \vert\vert\vert^{2}_{\omega}
\end{align*}
where $\lambda_{0}=\frac{b_{1}}{b_{2}}\lambda_{x}$ and
\begin{itemize}
  \item
   $x_{lu}\in \{(1-\lambda_{x},y):\lambda_{0}\le y\le 1-\lambda_{y} \}$;
  \item
  $x_{ld}\in \{(1-\lambda_{x},y):0\le y \le\lambda_{0} \}$;
  \item
  $x_{u}\in \{(x,1-\lambda_{y}): 1-\lambda_{x}\le x \le 1 \}$.
\end{itemize}

\end{proof}

\begin{lemma}
If $\sigma_{\beta}=kN^{-1}\ln N$ and $\sigma_{\eta}=k\tilde{\varepsilon}^{1/2}\ln N$ , where $k>1$ sufficiently
large and independent of $N$ and $\varepsilon$. Then
\begin{equation*}
B((\omega^{-1}G)^{I}-\omega^{-1}G,G)\le \frac{1}{16}\Vert\vert G \vert\Vert^{2}_{\omega}.
\end{equation*}
\end{lemma}
\begin{proof}
Cauchy-Schwarz¡¯s inequality gives
\begin{align*}
\left|B(E,G)\right|
&\le
(\varepsilon+b^{2}\delta)^{1/2}\Vert \omega^{1/2}E_{\beta} \Vert\cdot
(\varepsilon+b^{2}\delta)^{1/2}\Vert \omega^{-1/2}G_{\beta} \Vert
+\hat{\varepsilon}^{1/2}\Vert \omega^{1/2}E_{\eta} \Vert\cdot
\hat{\varepsilon}^{1/2}\Vert \omega^{-1/2}G_{\eta} \Vert\\
&+C\Vert \omega^{1/2}E \Vert\cdot
\Vert \omega^{-1/2}G_{\beta} \Vert
+\Vert \omega^{1/2}E \Vert\cdot \Vert \omega^{-1/2}G \Vert
\end{align*}
From Lemma \ref{lemma derivative of E} and Lemma \ref{lemma E}, we are done.
\end{proof}

\newtheorem{theorem}{Theorem}[section]
\begin{theorem}
Assume that $\sigma_{\beta}=kN^{-1}\ln N$ and $\sigma_{\eta}=k\tilde{\varepsilon}^{1/2}\ln N$, where $k>0$
is sufficiently large and independent of $\varepsilon$ and $N$. Let $\boldsymbol{x}^{*}\in\Omega\setminus\Omega_{xy}$. Then
for each nonnegative integer $\upsilon$, there exists a positive constant $C=C(\upsilon)$ and $K=K(\upsilon)$ such that
\begin{equation*}
\Vert  G \Vert_{W^{1,\infty}(\Omega_{s}\setminus\Omega'_{0})} \le CN^{-\upsilon},
\end{equation*}
\begin{equation*}
 \varepsilon\vert  G \vert_{W^{1,\infty}((\Omega_{x}\cup\Omega_{y})\setminus\Omega'_{0})}+
\Vert G \Vert_{L^{\infty}((\Omega_{x}\cup\Omega_{y})\setminus\Omega'_{0})} \le CN^{-\upsilon}
\end{equation*}
and
\begin{equation*}
\varepsilon\vert  G \vert_{W^{1,\infty}(\Omega_{xy}\setminus\Omega^{\prime}_{0})}+
\Vert G \Vert_{L^{\infty}(\Omega_{xy}\setminus\Omega^{\prime}_{0})}
\le
C\varepsilon^{-1/2}N^{-\upsilon}.
\end{equation*}
\end {theorem}
\begin{proof}
On $\Omega_{s}$, we apply an inverse estimate.\\*
On $\Omega\setminus\Omega_{s}$ the application of an
inverse estimate does not yield a satisfactory result, so we use a different technique.
\par
Let $\boldsymbol{x}\in\Omega_{x}\setminus\Omega^{\prime}_{0}$ be arbitrary. Starting from $\boldsymbol{x}$ we choose a polygonal curve $\Gamma\subset(\Omega\setminus\Omega_{xy})\setminus\Omega^{\prime}_{0}$ that joints $\boldsymbol{x}$ with some point on outflow boundaries. If $(\boldsymbol{x}-\boldsymbol{x}^{\ast})\cdot \eta<0$, we can choose $\Gamma$ as a line parallel to $\beta$. If $(\boldsymbol{x}-\boldsymbol{x}^{\ast})\cdot \eta>0$, the situation is a little complicated. We can choose $\Gamma$ as follows:\\*
In $\Omega_{x}\setminus\Omega^{\prime}_{0}$, we choose the direction of $\Gamma$ along $\eta$ or the negative direction of $x$-axis so that $\Gamma\cap\Omega_{xy}=\Phi$. In $(\Omega_{s}\cup\Omega_{y})\setminus\Omega^{\prime}_{0}$, we choose the direction of $\Gamma$ along $\eta$ or the positive direction of $y-$axis. \\*
 Let $T^{N}$ be the set of mesh rectangle $\tau$ in $(\Omega\setminus\Omega_{xy})\setminus\Omega^{\prime}_{0}$ that $\Gamma$ intersects. Note that the length of the segment of $\Gamma$ that lies in each $\tau$ is at most $C\varepsilon N^{-1}\ln N$ if $\tau\in \Omega_{x}$ or $\tau\in\Omega_{y}$.
\par
Then, by the fundamental theorem of calculus and inverse estimates in different domain, we can obtain the results.
\end{proof}

\bibliographystyle{plain}
\bibliography{alex-FE-paper,alex-FE-book}

\begin{thebibliography}{1}

\bibitem{Apel:1999-Anisotropic}
T.~Apel.
\newblock {\em Anisotropic Finite Elements: Local Estimates and Applications}.
\newblock B. G. Teubner Verlag, Stuttgart, 1999.

\bibitem{Apel1Dobr2:1992-Anisotropic}
T.~Apel and M.~Dobrowolski.
\newblock Anisotropic interpolation with applications to the finite element
  method.
\newblock {\em Computing}, 47:277--293, 1992.

\bibitem{Johnson:1987-Numerical}
C.~Johnson.
\newblock {\em Numerical Solution of Partial Differential Equations by the
  Finite Element Method}.
\newblock Cambridge University Press, Cambridge, 1987.

\bibitem{Joh1Sch2Wah3:1987-Crosswind}
C.~Johnson, A.~H. Schatz, and L.~B. Wahlbin.
\newblock Crosswind smear and pointwise errors in streamline diffusion finite
  element methods.
\newblock {\em Mathematics of Computation}, 49(179):25--38, 1987.

\bibitem{Linb1Styn2:2001-SDFEM}
T.~Lin{\ss} and M.~Stynes.
\newblock The {SDFEM} on shishkin meshes for linear convection--diffusion
  problems.
\newblock {\em Numer. Math.}, 87:457--484, 2001.

\end{thebibliography}
\end{document}